\documentclass[12pt,a4paper]{article}
\usepackage{amsmath,amsthm,amssymb}
\usepackage[truedimen,margin=25truemm]{geometry}
\usepackage{mathrsfs}
\usepackage{here}
\usepackage{bm}
\usepackage{graphicx}
\usepackage[utf8]{inputenc}

\theoremstyle{definition}
\newtheorem{definition}{定義}
\newtheorem{thm}[definition]{Theorem}

\newtheorem{lem}[definition]{Lemma}
\newtheorem{cor}[definition]{Corollary}

\newcommand{\diag}{{\rm diag}}

\DeclareMathOperator{\SL}{SL}

\DeclareMathOperator{\Ker}{Ker}
\DeclareMathOperator{\Ima}{Im}
\DeclareMathOperator{\Span}{span}

\begin{document}

\title{First and Second Fundamental Theorems for Invariant Rings Generated by Circulant Determinants}
\author{Naoya Yamaguchi, Hiroyuki Ochiai, and Yuka Yamaguchi}
\date{\today}

\maketitle

\begin{abstract}
In this paper, 
we give the first and second fundamental theorems of invariant theory for certain invariant rings whose generators are expressed by circulant determinants. 
\end{abstract}

\section{Introduction}
The first and second fundamental theorems of invariant theory concern the generators of the invariant ring and the relations among these generators, where the first provides the generators and the second provides the relations. 
The generators and their relations are central topics in the study of invariant rings. 
As tools for explicit descriptions, determinants (e.g., discriminants, resultants, and Hessians) and partial differential operators such as the Cayley-Aronhold operators play an important role. 
It is well known that Hermann Weyl established the first and second fundamental theorems for classical groups \cite{MR1488158}. 
Not many invariant rings have yet been obtained for which the generators and their relations can be explicitly written down, and even today, active research continues to give the first and second fundamental theorems for various invariant rings (see, e.g., \cite{MR894702, MR1137542, MR1848969, MR3436396, MR4104497}). 

For any positive integer $n$, 
let $\bm{x} := {}^t (x_{0}, x_{1}, \ldots, x_{n - 1})$ be an $n$-tuple of variables and let $R[\bm{x}] := R[x_{0}, x_{1}, \ldots, x_{n-1}] = R[x_{i} \,\vert\, 0 \leq i \leq n - 1]$ denote the polynomial ring in the $n$ variables over a ring $R$. 
We consider the following differential operators: 
\begin{align*}
D := \sum_{i = 0}^{n-1} x_{i-1} \frac{\partial}{\partial x_{i}}, \quad 
\Delta := \sum_{i = 0}^{n-1} x_{i+1} \frac{\partial}{\partial x_{i}} 
\end{align*}
with the conventions $x_{-1} := x_{n - 1}$ and $x_{n} := x_{0}$. 
These are special cases of the differential operators introduced by Frobenius \cite[p.~50]{Frobenius1968gruppen}, \cite{MR0446837}, \cite[p.~226]{MR803326} to study the group determinant, 
which is a generalization of the circulant determinant. 
The above operators derive the following invariant rings: 
\begin{align*}
\mathbb{C}[\bm{x}]^{D} := \{ f(\bm{x}) \in \mathbb{C}[\bm{x}] \mid D f(\bm{x}) = 0 \}, \quad 
\mathbb{C}[\bm{x}]^{\Delta} := \{ f(\bm{x}) \in \mathbb{C}[\bm{x}] \mid \Delta f(\bm{x}) = 0 \}. 
\end{align*}
Let $\Theta_{n}(\bm{x})$ denote the circulant determinant of order $n$; 
that is, 
\[ 
\Theta_{n}(\bm{x}) := 
\det{\begin{pmatrix} x_{0} & x_{n - 1} & \cdots & x_{1} \\ x_{1} & x_{0} & \cdots & x_{2} \\ \vdots & \vdots & \ddots & \vdots \\ x_{n - 1} & x_{n - 2} & \cdots & x_{0} \end{pmatrix}} \in \mathbb{Z}[\bm{x}]. 
\]
Let $n_{p}$ denote $\frac{n}{p}$ for a positive divisor $p$ of $n$. 
For $0 \leq i \leq n_{p} - 1$, 
we define 
\begin{align*}
\Theta_{p}(\bm{y}^{(p)}_{i}) 
:= \det{\begin{pmatrix} y^{(p)}_{i, 0} & y^{(p)}_{i, p - 1} & \cdots & y^{(p)}_{i, 1} \\ y^{(p)}_{i, 1} & y^{(p)}_{i, 0} & \cdots & y^{(p)}_{i, 2} \\ \vdots & \vdots & \ddots & \vdots \\ y^{(p)}_{i, p - 1} & y^{(p)}_{i, p - 2} & \cdots & y^{(p)}_{i, 0} \end{pmatrix}}, \quad 
y^{(p)}_{i, j} := \sum_{l = 0}^{n_{p} - 1} \zeta_{n}^{i (j + p l)} x_{j + p l} \quad (0 \leq j \leq p - 1), 
\end{align*}
where $\zeta_{n}$ is a primitive $n$-th root of unity. 
In this paper, 
we give the first and second fundamental theorems of invariant theory for $\mathbb{C}[\bm{x}]^{D}$ and $\mathbb{C}[\bm{x}]^{\Delta}$ in the case that $n$ has at most two prime factors. 

\begin{thm}[First Fundamental Theorem]\label{thm:1}
%Let $p$ and $q$ be prime numbers and let $n = p^{k} q^{l}$, 
%where $k$ and $l$ are non-negative integers. 
%Then, the following holds: 
Let $n$ be a positive integer. 
Then, the following holds: 
If $n$ has exactly one prime factor $p$, then 
\begin{align*}
\mathbb{C}[\bm{x}]^{D} = \mathbb{C}[\bm{x}]^{\Delta} = \mathbb{C}[\Theta_{p}(\bm{y}^{(p)}_{i}) \,\vert\, 0 \leq i \leq n_{p} - 1]; 
%= \mathbb{C}[\Theta_{p}(\bm{y}^{(p)}_{0}), \Theta_{p}(\bm{y}^{(p)}_{1}), \ldots, \Theta_{p}(\bm{y}^{(p)}_{n_{p} - 1})]; 
\end{align*}
If $n$ has exactly two prime factors $p$ and $q$, then 
\begin{align*}
\mathbb{C}[\bm{x}]^{D} = \mathbb{C}[\bm{x}]^{\Delta} = \mathbb{C}[\Theta_{p}(\bm{y}^{(p)}_{i}), \Theta_{q}(\bm{y}^{(q)}_{j}) \,\vert\, 0 \leq i\leq n_{p} - 1,\, 0 \leq j \leq n_{q} - 1]. 
%= \mathbb{C}[\Theta_{p}(\bm{y}^{(p)}_{0}), \Theta_{p}(\bm{y}^{(p)}_{1}), \ldots, \Theta_{p}(\bm{y}^{(p)}_{n_{p} - 1}), \Theta_{q}(\bm{y}^{(q)}_{0}), \Theta_{q}(\bm{y}^{(q)}_{1}), \ldots, \Theta_{q}(\bm{y}^{(q)}_{n_{q} - 1})]. 
\end{align*}
\end{thm}

\begin{thm}[Second Fundamental Theorem]\label{thm:2}
Let $n$ be a positive integer, and let $p$ and $q$ be coprime positive integers satisfying $p q \mid n$. 
For arbitrary variables $z_{i}$ ($0 \leq i \leq n_{p} - 1$) and $w_{j}$ ($0 \leq j \leq n_{q} - 1$), 
we consider the ring homomorphisms 
\begin{align*}
\rho : \mathbb{C}[z_{i} \,\vert\, 0 \leq i \leq n_{p} - 1] \to \mathbb{C}[\Theta_{p}(\bm{y}^{(p)}_{i}) \,\vert\, 0 \leq i \leq n_{p} - 1]
\end{align*}
defined by $\rho (z_{i}) := \Theta_{p}(\bm{y}^{(p)}_{i})$; 
\begin{align*}
\rho' &: \mathbb{C}[z_{i}, w_{j} \,\vert\, 0 \leq i \leq n_{p} - 1,\, 0 \leq j \leq n_{q} - 1] \\
&\quad \to \mathbb{C}[\Theta_{p}(\bm{y}^{(p)}_{i}), \Theta_{q}(\bm{y}^{(q)}_{j}) \,\vert\, 0 \leq i \leq n_{p} - 1,\, 0 \leq j \leq n_{q} - 1]
\end{align*}
defined by $\rho' (z_{i}) := \Theta_{p}(\bm{y}^{(p)}_{i})$ and $\rho' (w_{j}) := \Theta_{q}(\bm{y}^{(q)}_{j})$. 
Then, the following holds: 
\begin{align*} 
\Ker(\rho) = \{0 \}; \quad 
\Ker(\rho') = (t_{0}, t_{1}, \ldots, t_{n_{p q} - 1}), 
\end{align*}
where $t_{i} := \prod_{j = 0}^{q - 1} z_{i + n_{pq} j} - \prod_{j = 0}^{p - 1} w_{i + n_{pq} j}$. 
\end{thm}

Also, we demonstrate that the equality analogous to Theorem~$\ref{thm:1}$ does not hold when $n$ has at least three prime factors. 

\begin{thm}\label{thm:3}
Let $n$ be a positive integer with at least three prime factors. 
Then, we have 
\[
\mathbb{C}[\bm{x}]^{D} = \mathbb{C}[\bm{x}]^{\Delta} \supsetneq \mathbb{C}[\Theta_{p}(\bm{y}^{(p)}_{i}) \mid \text{$p$ is a prime factor of $n$},\; 0 \leq i \leq n_{p} - 1]. 
\]
\end{thm}

With respect to the circulant determinants, we prove the following theorem, which is a specialization of both Laquer's theorem~\cite{MR4526227} and the generalized Dedekind's theorem~\cite{MR4814714}. 

\begin{thm}\label{thm:4}
For any positive integer $n$ and its positive divisor $p$, the following holds: 
\begin{align*}
\Theta_{n}(\bm{x}) = 
\prod_{i = 0}^{n_{p} - 1} \Theta_{p}(\bm{y}^{(p)}_{i}).  
%= \prod_{i = 0}^{\frac{n}{q} - 1} \Theta_{q}(\bm{y}^{(q)}_{i}).  
\end{align*}
\end{thm}

In addition, we give a theorem. 

\begin{thm}\label{thm:5}
Let $G$ denote the cyclic group of order $n \geq 2$ and let 
\begin{align*}
\SL(\mathbb{C}G) &:= 
\left\{ 
\begin{pmatrix} a_{0} & a_{n - 1} & \cdots & a_{1} \\ a_{1} & a_{0} & \cdots & a_{2} \\ \vdots & \vdots & \ddots & \vdots \\ a_{n - 1} & a_{n - 2} & \cdots & a_{0} \end{pmatrix} \mid a_{0}, a_{1}, \ldots, a_{n - 1} \in \mathbb{C}
\right\} \cap \SL(n, \mathbb{C}), \\ 
\mathbb{C}[\bm{x}]^{\SL(\mathbb{C}G)} &:= \left\{ f(\bm{x}) \in \mathbb{C}[\bm{x}] \mid  f(A \bm{x}) = f(\bm{x}) \: \: \text{for any} \: A \in \SL(\mathbb{C}G) \right\}. 
\end{align*}
Then it holds that 
\[
\mathbb{C}[\bm{x}]^{\SL(\mathbb{C}G)} = \mathbb{C}[\Theta_{n}(\bm{x})] \subset \mathbb{C}[\bm{x}]^{D}. 
\]
\end{thm}

From Theorems~$\ref{thm:1}$ and $\ref{thm:5}$, 
we can immediately obtain the following corollary. 

\begin{cor}
For a group $G$ of prime order, 
it holds that 
$\mathbb{C}[\bm{x}]^{\SL(\mathbb{C}G)} = \mathbb{C}[\bm{x}]^{D} = \mathbb{C}[\bm{x}]^{\Delta}$. 
\end{cor}

\section{Preliminaries}
\subsection{Variable transformation}
We consider the variable vector $\bm{x} := {}^{t} (x_{0}, x_{1}, \ldots, x_{n - 1})$ and perform a linear variable transformation to a  new variable vector $\bm{y} := {}^{t}(y_{0}, y_{1}, \ldots, y_{n-1})$ defined as 
\[
y_{i} := \sum_{j = 0}^{n - 1} \zeta_{n}^{i j} x_{j} \quad (0 \leq i \leq n - 1), 
\]
which are the eigenvalues of the $n$-dimensional circulant matrix composed of the vector~$\bm{x}$. 
Since 
\begin{align*}
D(y_{i}) &= \sum_{j = 0}^{n - 1} \zeta_{n}^{i j} D(x_{j}) = \sum_{j = 0}^{n - 1} \zeta_{n}^{i j} x_{j-1} = \sum_{j = 0}^{n - 1} \zeta_{n}^{i (j + 1)} x_{j} = \zeta_{n}^{i} y_{i}, \\ 
\Delta(y_{i}) &= \sum_{j = 0}^{n - 1} \zeta_{n}^{i j} \Delta(x_{j}) = \sum_{j = 0}^{n - 1} \zeta_{n}^{i j} x_{j+1} = \sum_{j = 0}^{n - 1} \zeta_{n}^{i (j - 1)} x_{j} = \zeta_{n}^{- i} y_{i}
\end{align*}
hold (These equations are special cases of \cite[p.~227, (8)]{MR0446837}), 
we can express the differential operators as follows:
\begin{align*}
D = \sum_{i = 0}^{n - 1} \zeta_{n}^{i} y_{i} \frac{\partial}{\partial y_{i}}, \quad 
\Delta = \sum_{i = 0}^{n - 1} \zeta_{n}^{- i} y_{i} \frac{\partial}{\partial y_{i}}. 
\end{align*}
For an $n$-tuple $\bm{\alpha} := (\alpha_{0}, \alpha_{1}, \ldots, \alpha_{n - 1})$ of non-negative integers, 
we define 
\[
\bm{y}^{\bm{\alpha}} := y_{0}^{\alpha_{0}} y_{1}^{\alpha_{1}} \cdots y_{n - 1}^{\alpha_{n - 1}}. 
\]
The differential operators $D$ and $\Delta$ act on the monomial $\bm{y}^{\bm{\alpha}}$ by scalar multiplication with the scalar multipliers $\sum_{i = 0}^{n - 1} \alpha_{i} \zeta_{n}^{i}$ and $\sum_{i = 0}^{n - 1} \alpha_{i} \zeta_{n}^{- i}$, respectively, which are complex conjugates of each other. 
Therefore, the invariant rings $\mathbb{C}[\bm{y}]^{D}$ and $\mathbb{C}[\bm{y}]^{\Delta}$ can be regarded as the vector spaces over $\mathbb{C}$ generated by the monomials $\bm{y}^{\bm{\alpha}}$ corresponding to tuples $\bm{\alpha}$ for which the scalar multipliers are zero. 
Thus, we have 
\[
\mathbb{C}[\bm{y}]^{D} = \mathbb{C}[\bm{y}]^{\Delta} = \mathbb{C}\mathchar`-\Span\{ \bm{y}^{\bm{\alpha}} \mid \bm{\alpha} \in V_{n} \cap (\mathbb{Z}_{\geq 0})^{n} \}, 
\]
where $V_{n}$ is the vector space over $\mathbb{Q}$ defined as 
\[
V_{n} := \{ (\alpha_{0}, \alpha_{1}, \ldots, \alpha_{n - 1}) \in \mathbb{Q}^{n} \mid \sum_{i = 0}^{n - 1} \alpha_{i} \zeta_{n}^{i} = 0 \}. 
\]

\subsection{Exponent vectors}
Let $\bm{e}_{0}, \bm{e}_{1}, \ldots, \bm{e}_{n - 1}$ be the standard basis of the vector space $\mathbb{Q}^{n}$. 

\begin{lem}\label{lem:7}
The dimension of the vector space $V_{n}$ over $\mathbb{Q}$ is $n - \varphi(n)$, 
where $\varphi$ denotes the Euler's totient function. 
\end{lem}
\begin{proof}
When $n = 1$, the statement of the lemma is true since the dimension of $V_{1} = \{ \bm{0} \}$ is zero. 
Suppose that $n \geq 2$ and let $\sum_{j = 0}^{\varphi(n)} b_{j} x^{j} \in \mathbb{Z}[x]$ be the $n$-th cyclotomic polynomial. 
We define 
\[
\bm{v}_{i} := \sum_{j = 0}^{\varphi(n)} b_{j} \bm{e}_{i + j} \quad (0 \leq i \leq n - \varphi(n) - 1)
\]
so that $\bm{v}_{i} \in V_{n}$ holds for each $i$. 
Then, the set 
\[
S_{n} := \{ \bm{v}_{i} \mid 0 \leq i \leq n - \varphi(n) - 1 \}
\]
forms a basis of $V_{n}$. 
This can be verified as follows: 
(i) If $\sum_{i = 0}^{n - \varphi(n) - 1} c_{i} \bm{v}_{i} = \bm{0}$ for some $c_{i} \in \mathbb{Q}$, then we have $c_{0} = c_{1} = \cdots = c_{n - \varphi(n) - 1} = 0$ from 
\[
\sum_{i = 0}^{n - \varphi(n) - 1} c_{i} \bm{v}_{i} 
= \sum_{i = 0}^{n - \varphi(n) - 1} c_{i} \sum_{j = 0}^{\varphi(n)} b_{j} \bm{e}_{i + j} 
= \sum_{i = 0}^{n - 1} \sum_{j = \max{\{i - (n - \varphi(n) - 1),\, 0 \}}}^{\min{\{i,\, \varphi(n)\}}} b_{j} c_{i - j} \bm{e}_{i}
\]
and $b_{0} \neq 0$. 
Thus, $S_{n}$ is linearly independent. 
(ii) Since the cyclotomic polynomial is the minimal polynomial of $\zeta_{n}$ over $\mathbb{Q}$, 
for any $\bm{\alpha} = (\alpha_{0}, \alpha_{1}, \ldots, \alpha_{n-1}) \in V_{n}$, there exists $\sum_{i = 0}^{n - \varphi(n) - 1} c_{i} x^{i} \in \mathbb{Q}[x]$ satisfying 
\[
\sum_{i = 0}^{n - 1} \alpha_{i} x^{i} = \sum_{j = 0}^{\varphi(n)} b_{j} x^{j} \sum_{i = 0}^{n - \varphi(n) - 1} c_{i} x^{i} = \sum_{i = 0}^{n - 1} \sum_{j = \max{\{i - (n - \varphi(n) - 1),\, 0 \}}}^{\min{\{i,\, \varphi(n)\}}} b_{j} c_{i - j} x^{i}. 
\]
By comparing the coefficients, we have $\alpha_{i} = \sum_{j = \max{\{i - (n - \varphi(n) - 1),\, 0 \}}}^{\min{\{i,\, \varphi(n)\}}} b_{j} c_{i - j}$, and therefore  
\[
\sum_{i = 0}^{n - \varphi(n) - 1} c_{i} \bm{v_{i}} 
= \sum_{i = 0}^{n - 1} \sum_{j = \max{\{i - (n - \varphi(n) - 1),\, 0 \}}}^{\min{\{i,\, \varphi(n)\}}} b_{j} c_{i - j} \bm{e}_{i} 
= \sum_{i = 0}^{n - 1} \alpha_{i} \bm{e}_{i} 
= \bm{\alpha}. 
\]
This means that $S_{n}$ spans $V_{n}$. 
\end{proof}

Since the coefficients of the cyclotomic polynomial can be negative, 
the set $S_{n}$ is not necessarily suitable for describing the non-negative part of $V_{n}$. 
Now, for each positive divisor $p$ of $n$, 
we define 
\[
{\bm v}^{(p)}_{i} := \sum_{j = 0}^{p - 1} \bm{e}_{i + n_{p} j} \in (\mathbb{Z}_{\geq 0})^{n} \quad (0 \leq i \leq n_{p} - 1).  
\]

\begin{lem}\label{lem:8}
For any divisor $p \geq 2$ of $n$, 
we have ${\bm v}^{(p)}_{i} \in V_{n} \cap (\mathbb{Z}_{\geq 0})^{n}$ for $0 \leq i \leq n_{p} - 1$. 
\end{lem}
\begin{proof}
It follows from 
$\sum_{j = 0}^{p - 1} \zeta_{n}^{i + n_{p} j} = \zeta_{n}^{i} \sum_{j = 0}^{p - 1} \zeta_{n}^{n_{p} j} = \zeta_{n}^{i} \sum_{j = 0}^{p - 1} \zeta_{p}^{j} = 0$ 
that ${\bm v}^{(p)}_{i} \in V_{n}$.  
\end{proof}

We define the set $T_{n}$ as follows: 
\[
T_{n} := \bigcup_{\substack{p \mid n \\ p: \: \text{prime}}} \{ {\bm v}^{(p)}_{i} \mid 0 \leq i \leq n_{p} - 1 \}.  
\]
From Lemma~$\ref{lem:8}$, 
we have 
\[  
V_{n} \cap (\mathbb{Z}_{\geq 0})^{n} \supset 
\mathbb{Z}_{\geq 0}\mathchar`-\Span T_{n}. 
\]
When $n$ has at most two prime factors, 
the reverse inclusion also holds as shown in Section~$3.1$. 

\begin{lem}\label{lem:9}
Suppose that $n$ has at least three prime factors. Then, it holds that 
\[
V_{n} \cap (\mathbb{Z}_{\geq 0})^{n} \supsetneq 
\mathbb{Z}_{\geq 0}\mathchar`-\Span T_{n}. 
\]
\end{lem}
\begin{proof}
Let $n = p_{1}^{l_1} p_{2}^{l_2} \cdots p_{k}^{l_{k}}$ be the prime factorization of $n$, where $k \geq 3$ and $l_j \geq 1$. 
Note that the following holds: If $\bm{\alpha} := (\alpha_{0}, \alpha_{1}, \ldots, \alpha_{n - 1}) \in \mathbb{Z}_{\geq 0}\mathchar`-\Span T_{n}$ with $\alpha_{0} \geq 1$, then at least one of $\alpha_{n_{p_1}}$, $\alpha_{n_{p_2}}$, \ldots, $\alpha_{n_{p_k}}$ is not zero. 
For any positive divisor $d$ of $n$, we define ${\bm v}^{(d)}_{i + n_{d} j} := {\bm v}^{(d)}_{i}$ for $0 \leq i \leq n_{d} - 1$ and $j \in \mathbb{Z}$. 
Now, take any $p, q, r \in \{p_{1}, p_{2}, \ldots, p_{k} \}$ with $p < q < r$, and let $\bm{\alpha}' = (\alpha'_{0}, \alpha'_{1}, \ldots, \alpha'_{n - 1}) := \sum_{j = 2}^{p} {\bm v}^{(q)}_{n_{p} j} + {\bm v}^{(r)}_{n_p + n_q} - {\bm v}^{(p)}_{n_q}$. 
Then, since $\alpha'_{0} = 1$ and $\alpha'_{n_{p_1}} = \alpha'_{n_{p_2}} = \cdots = \alpha'_{n_{p_k}} = 0$ hold, we have $\bm{\alpha}' \notin \mathbb{Z}_{\geq 0}\mathchar`-\Span T_{n}$. 
On the other hand, we have $\bm{\alpha}' \in V_{n} \cap (\mathbb{Z}_{\geq 0})^{n}$ since $\bm{\alpha}' \in V_{n}$ holds from Lemma~$\ref{lem:8}$ and $\bm{\alpha}' \in (\mathbb{Z}_{\geq 0})^{n}$ holds from the definition. 
Thus the lemma is proved. 
\end{proof}

\begin{lem}\label{lem:10}
For any positive integers $p$ and $q$ satisfying $p q \mid n$, 
it holds that 
\[
\sum_{j = 0}^{q - 1} {\bm v}^{(p)}_{i + n_{pq} j} 
= \sum_{j = 0}^{p - 1} {\bm v}^{(q)}_{i + n_{pq} j}
\]
for $0 \leq i \leq n_{p q} - 1$. 
\end{lem}
\begin{proof}
For any $0 \leq i \leq n_{pq} - 1$, we have 
\begin{align*}
\sum_{j = 0}^{q - 1} {\bm v}^{(p)}_{i + n_{pq} j} 
= \sum_{j = 0}^{q - 1} \sum_{k = 0}^{p - 1} \bm{e}_{i + n_{pq} (j + q k)} 
= \sum_{j = 0}^{p q - 1} \bm{e}_{i + n_{p q} j}. 
\end{align*}
Since the last expression is symmetric in $p$ and $q$, 
we obtain the lemma. 
\end{proof}

The following lemma implies that, when positive integers $p$ and $q$ satisfying $p q \mid n$ are coprime, all the linear relations among ${\bm v}^{(p)}_{i}$ ($0 \leq i \leq n_{p} - 1$) and ${\bm v}^{(q)}_{i}$ ($0 \leq i \leq n_{q} - 1$) can be derived from the relations in Lemma~$\ref{lem:10}$. 

\begin{lem}\label{lem:11}
Let $p$ and $q$ be coprime positive integers satisfying $p q \mid n$. 
If 
\[
\sum_{i = 0}^{n_{p q} - 1} \sum_{j = 0}^{q - 1} c_{i, j} {\bm v}_{i + n_{p q} j}^{(p)} = \sum_{i = 0}^{n_{p q} - 1} \sum_{j = 0}^{p - 1} d_{i, j} {\bm v}_{i + n_{p q} j}^{(q)}
\]
for some $c_{i, j}, d_{i, j} \in \mathbb{Q}$, then $c_{i, 0} = c_{i, 1} = \cdots = c_{i, q -1} = d_{i, 0} = d_{i, 1} = \cdots = d_{i, p - 1}$ holds for $0 \leq i \leq n_{p q} - 1$. 
\end{lem}
\begin{proof}
Suppose that $\sum_{i = 0}^{n_{p q} - 1} \sum_{j = 0}^{q - 1} c_{i, j} {\bm v}_{i + n_{p q} j}^{(p)} = \sum_{i = 0}^{n_{p q} - 1} \sum_{j = 0}^{p - 1} d_{i, j} {\bm v}_{i + n_{p q} j}^{(q)}$ holds, and fix any $0 \leq i \leq n_{p q} - 1$. 
Then, from 
\begin{align*}
\sum_{j = 0}^{q - 1} c_{i, j} {\bm v}_{i + n_{p q} j}^{(p)} 
&= \sum_{j = 0}^{q - 1} \sum_{k = 0}^{p - 1} c_{i, j} \bm{e}_{i + n_{pq} (j + q k)} 
= \sum_{j = 0}^{p q - 1} c_{i, \mathrm{rem}_{q}(j)} \bm{e}_{i + n_{p q} j}, \\
\sum_{j = 0}^{p - 1} d_{i, j} {\bm v}_{i + n_{p q} j}^{(q)} 
&= \sum_{j = 0}^{p - 1} \sum_{k = 0}^{q - 1} d_{i, j} \bm{e}_{i + n_{pq} (j + p k)} 
= \sum_{j = 0}^{p q - 1} d_{i, \mathrm{rem}_{p}(j)} \bm{e}_{i + n_{p q} j}, 
\end{align*}
where $\mathrm{rem}_{a}(j)$ denotes the reminder when $j$ is divided by $a$, 
we have $c_{i, \mathrm{rem}_{q}(j)} = d_{i, \mathrm{rem}_{p}(j)}$ for $0 \leq j \leq p q - 1$. 
From this, for any $0 \leq k \leq q - 1$, it holds that $c_{i, \mathrm{rem}_{q}(p k)} = d_{i, \mathrm{rem}_{p}(p k)} = d_{i, 0}$. 
Therefore, we have $c_{i, 0} = c_{i, 1} = \cdots = c_{i, q - 1} = d_{i, 0}$ since $\{\mathrm{rem}_{q} (p k) \mid 0 \leq k \leq q - 1 \} = \{0, 1, \ldots, q - 1 \}$ holds from $\gcd{(p, q)} = 1$. 
Similarly, we can prove $d_{i, 0} = d_{i, 1} = \cdots = d_{i, p - 1} = c_{i, 0}$. 
\end{proof}

For any positive divisors $p$ and $q$ of $n$, 
let $\sigma : \mathbb{Z}^{n_p} \times \mathbb{Z}^{n_q} \to \mathbb{Z}^{n}$ be the homomorphism of the $\mathbb{Z}$-modules defined by 
\[
\sigma (\bm{c}, \bm{d}) := \sum_{i = 0}^{n_{p} - 1} c_{i} {\bm v}^{(p)}_{i} + \sum_{i = 0}^{n_{q} - 1} d_{i} {\bm v}^{(q)}_{i}, 
\]
where $\bm{c} = (c_{0}, c_{1}, \ldots, c_{n_{p} - 1})$ and $\bm{d} = (d_{0}, d_{1}, \ldots, d_{n_{q} - 1})$. 
From Lemma~$\ref{lem:11}$, we have the following lemma. 

\begin{lem}\label{lem:12}
Let $p$ and $q$ be coprime positive integers satisfying $p q \mid n$. 
Then, we have 
\begin{align*}
\Ker(\sigma) = \{(\bm{c}, \bm{d}) \in \mathbb{Z}^{n_p} \times \mathbb{Z}^{n_q} \mid \bm{c} = (\underbrace{\bm{m}, \bm{m}, \ldots, \bm{m}}_{q}),\, \bm{d} = - (\underbrace{\bm{m}, \bm{m}, \ldots, \bm{m}}_{p}),\, \bm{m} \in \mathbb{Z}^{n_{p q}} \}. 
\end{align*}
\end{lem}

\subsection{Polynomial ring generated by circulant determinants}

\begin{lem}\label{lem:13}
For any positive divisor $p$ of $n$, we have $\bm{y}^{{\bm v}^{(p)}_{i}} = \Theta_{p}(\bm{y}^{(p)}_{i})$ for $0 \leq i \leq n_{p} - 1$.
\end{lem}
\begin{proof}
It follows from 
\begin{align*}
y_{i + n_{p} j} 
= \sum_{k = 0}^{n - 1} \zeta_{n}^{(i + n_{p} j) k} x_{k} 
= \sum_{k = 0}^{p - 1} \sum_{l = 0}^{n_{p} - 1} \zeta_{n}^{(i + n_{p} j) (k + p l)} x_{k + p l} 
%= \sum_{k = 0}^{p - 1} \zeta_{p}^{j k} \sum_{l = 0}^{n_{p} - 1} \zeta_{n}^{i (k + p l)} x_{k + p l} 
= \sum_{k = 0}^{p - 1} \zeta_{p}^{j k} y^{(p)}_{i, k}
\end{align*}
that 
$\bm{y}^{{\bm v}^{(p)}_{i}} 
= \prod_{j = 0}^{p - 1} y_{i + n_{p} j} 
= \prod_{j = 0}^{p - 1} \sum_{k = 0}^{p - 1} \zeta_{p}^{j k} y^{(p)}_{i, k} 
= \Theta_{p}(\bm{y}^{(p)}_{i})$. 
\end{proof}

We define the polynomial ring $R_{n}$ by 
\[
R_{n} := \mathbb{C}[\Theta_{p}(\bm{y}^{(p)}_{i}) \mid \text{$p$ is a prime factor of $n$},\; 0 \leq i \leq n_{p} - 1]. 
\]
As seen in Section~$2.1$, it holds that 
\[
\mathbb{C}[\bm{x}]^{D} = \mathbb{C}[\bm{x}]^{\Delta} = \mathbb{C}\mathchar`-\Span\{ \bm{y}^{\bm{\alpha}} \mid \bm{\alpha} \in V_{n} \cap (\mathbb{Z}_{\geq 0})^{n} \}. 
\]
Therefore, from Lemmas~$\ref{lem:8}$ and $\ref{lem:13}$, 
we have the following lemma. 
\begin{lem}\label{lem:14}
For any integer $n \geq 2$, it holds that $\mathbb{C}[\bm{x}]^{D} = \mathbb{C}[\bm{x}]^{\Delta} \supset R_{n}$. 
\end{lem}

\section{Proofs of theorems}
\subsection{Proof of Theorem~$\mathbf{\ref{thm:1}}$}
We prove Theorem~$\ref{thm:1}$. 
%From here to the end of this section, we assume that 
From Section~$2.3$, we find that to establish Theorem~$\ref{thm:1}$, 
it is sufficient to prove that the following holds: 
Take any $\bm{\alpha}  = (\alpha_{0}, \alpha_{1}, \ldots, \alpha_{n-1}) \in V_{n} \cap (\mathbb{Z}_{\geq 0})^{n}$. 
\begin{enumerate}
\item[(i)] If $n$ has exactly one prime factor $p$, then there exist non-negative integers $c_{i}$ satisfying $\bm{\alpha} = \sum_{i = 0}^{n_{p} - 1} c_{i} {\bm v}^{(p)}_{i}$. 
\item[(ii)] If $n$ has exactly two prime factors $p$ and $q$, then there exist non-negative integers $c_{i}$ and $d_{i}$ satisfying $\bm{\alpha} = \sum_{i = 0}^{n_{p} - 1} c_{i} {\bm v}^{(p)}_{i} + \sum_{i = 0}^{n_{q} - 1} d_{i} {\bm v}^{(q)}_{i}$. 
\end{enumerate}
That (i) is true can be easily verified as follows: 
It follows from Lemma~$\ref{lem:7}$ that, when $n$ has exactly one prime factor $p$, the vectors ${\bm v}^{(p)}_{i}$ ($0 \leq i \leq n_{p} - 1$) form a basis of the vector space $V_{n}$ over $\mathbb{Q}$ since they are linearly independent. 
Therefore, the coordinates $c_{i}$ ($0 \leq i \leq n_{p} - 1$) of $\bm{\alpha}$ with respect to the basis 
are uniquely determined as $c_{i} = \alpha_{i} \in \mathbb{Z}_{\geq 0}$. 

We prove that (ii) is true. 
Suppose that $n$ has exactly two prime factors $p$ and $q$. 
\begin{lem}\label{lem:15}
The vectors ${\bm v}^{(p)}_{i} \: (0 \leq i \leq n_{p} - 1)$ and ${\bm v}^{(q)}_{i} \: (0 \leq i \leq n_{q} - n_{pq} - 1)$ form a basis of the vector space $V_{n}$ over $\mathbb{Q}$. 
\end{lem}
\begin{proof}
From Lemma~$\ref{lem:7}$, the dimension of $V_{n}$ over $\mathbb{Q}$ is $n - \varphi(n) = n_{p} + n_{q} - n_{pq}$. 
Thus, from Lemma~$\ref{lem:8}$, it suffices to verify the linear independence of the vectors. 
Assume that for some $c_{i}, d_{i} \in \mathbb{Q}$, the following holds: 
\[
\sum_{i = 0}^{n_{p}  - 1} c_{i} {\bm v}^{(p)}_{i} + \sum_{i = 0}^{n_{q} - n_{pq} - 1} d_{i} {\bm v}^{(q)}_{i} = \bm{0}.
\]
Then, by letting $d_{i} := 0$ for $n_{q} - n_{p q} \leq i \leq n_{q} - 1$, we obtain
\[
\sum_{i = 0}^{n_{p q} - 1} \sum_{j = 0}^{q - 1} c_{i, j} {\bm v}_{i + n_{p q} j}^{(p)} 
= \sum_{i = 0}^{n_{p q} - 1} \sum_{j = 0}^{p - 1} d_{i, j} {\bm v}_{i + n_{p q} j}^{(q)}, 
\]
where $c_{i, j} := c_{i + n_{p q} j}$ and $d_{i, j} := - d_{i + n_{p q} j}$. 
Thus, from Lemma~$\ref{lem:11}$, it follows that $c_{i, 0} = c_{i, 1} = \cdots = c_{i, q -1} = d_{i, 0} = d_{i, 1} = \cdots = d_{i, p - 1} = - d_{i + n_{q} - n_{p q}} = 0$ for any $0 \leq i \leq n_{p q} - 1$. 
That is, it holds that $c_{0} = c_{1} = \cdots = c_{n_{p} - 1} = d_{0} = d_{1} = \cdots = d_{n_{q} - n_{p q} - 1} = 0$. 
\end{proof}

From Lemma~$\ref{lem:15}$, there exist $c_{i}, d_{i} \in \mathbb{Q}$ satisfying 
\begin{align*}
\bm{\alpha} 
= \sum_{i = 0}^{n_{p} - 1} c_{i} {\bm v}^{(p)}_{i} + \sum_{i = 0}^{n_{q} - n_{p q} - 1} d_{i} {\bm v}^{(q)}_{i} 
= \sum_{i = 0}^{n_{p q} - 1} \sum_{j = 0}^{q - 1} c_{i + n_{p q} j} {\bm v}_{i + n_{p q} j}^{(p)} + \sum_{i = 0}^{n_{p q} - 1} \sum_{j = 0}^{p - 2} d_{i + n_{p q} j} {\bm v}_{i + n_{p q} j}^{(q)}. 
\end{align*}
Let $c^{(i)} := \min{\{ c_{i + n_{p q} j} \mid 0 \leq j \leq q - 1 \} }$ for each $0 \leq i \leq n_{p q} - 1$. 
Then, from Lemma~$\ref{lem:10}$, 
we have 
\begin{align*}
%\sum_{i = 0}^{n_{p} - 1} c_{i} {\bm v}^{(p)}_{i} &= 
\sum_{i = 0}^{n_{p q} - 1} \sum_{j = 0}^{q - 1} c_{i + n_{p q} j} {\bm v}^{(p)}_{i + n_{p q} j} 
&= \sum_{i = 0}^{n_{p q} - 1} \sum_{j = 0}^{q - 1} (c_{i + n_{p q} j} - c^{(i)}) {\bm v}^{(p)}_{i + n_{p q} j} + \sum_{i = 0}^{n_{p q} - 1} c^{(i)} \sum_{j = 0}^{q - 1} {\bm v}^{(p)}_{i + n_{p q}j} \allowdisplaybreaks \\ 
&= \sum_{i = 0}^{n_{p q} - 1} \sum_{j = 0}^{q - 1} (c_{i + n_{p q} j} - c^{(i)}) {\bm v}^{(p)}_{i + n_{p q} j} + \sum_{i = 0}^{n_{p q} - 1} c^{(i)} \sum_{j = 0}^{p - 1} {\bm v}^{(q)}_{i + n_{p q} j} \allowdisplaybreaks \\
&= \sum_{i = 0}^{n_{p q} - 1} \sum_{j = 0}^{q - 1} (c_{i + n_{p q} j} - c^{(i)}) {\bm v}^{(p)}_{i + n_{p q} j} + \sum_{i = 0}^{n_{p q} - 1} \sum_{j = 0}^{p - 2} c^{(i)} {\bm v}^{(q)}_{i + n_{p q} j} \\
&\qquad + \sum_{i = 0}^{n_{pq} - 1} c^{(i)} {\bm v}^{(q)}_{i + n_{q} - n_{p q}}. 
\end{align*}
Therefore, we can express $\bm{\alpha}$ as 
\begin{align*}
\bm{\alpha} 
= \sum_{i = 0}^{n_{p q} - 1} \sum_{j = 0}^{q - 1} (c_{i + n_{p q} j} - c^{(i)}) {\bm v}^{(p)}_{i + n_{p q} j} 
+ \sum_{i = 0}^{n_{p q} - 1} \sum_{j = 0}^{p - 2} (c^{(i)} + d_{i + n_{p q} j} ) {\bm v}^{(q)}_{i + n_{p q} j} 
+ \sum_{i = 0}^{n_{pq} - 1} c^{(i)} {\bm v}^{(q)}_{i + n_{q} - n_{p q}}. 
\end{align*}

The following two lemmas complete the proof of Theorem~$\ref{thm:1}$. 

\begin{lem}\label{lem:16}
For any $0 \leq i \leq n_{p} - 1$, 
it holds that $c_{i} \in \mathbb{Z}_{\geq 0}$.  
\end{lem}
\begin{proof}
Let $I := \{(i, j) \mid n_{q} - n_{p q} \leq i \leq n_{q} - 1,\; 0 \leq j \leq q - 1\}$. 
We first prove that if $i + n_{q} j \equiv i' + n_{q} j' \pmod{n_{p}}$ for some $(i, j), (i', j') \in I$, then $i = i'$ and $j = j'$. 
Suppose that $i + n_{q} j = i' + n_{q} j' + n_{p} k$ holds for some integer $k$. Then, from $n_{p q} - 1 \geq \lvert i - i' \rvert = n_{p q} \lvert p (j' - j) + q k \rvert$ and $\gcd(p, q) = 1$, we have $j' - j = k = 0$, 
%since $p$ and $q$ are distinct prime numbers, 
and therefore we also have $i - i' = 0$. 
From the above, it holds that for any $0 \leq k \leq n_{p} - 1$, there exists $(i, j) \in I$ satisfying $k \equiv i + n_{q} j \pmod{n_{p}}$ since the cardinality of $I$ is $n_{p}$. 
From this and the fact that for every $(i, j) \in I$, the $(i + n_{q} j)$-th component of the vector $\sum_{i = 0}^{n_{q} - n_{pq} - 1} d_{i} {\bm v}^{(q)}_{i}$ is zero, we find that for any $0 \leq k \leq n_{p} - 1$, there exists $(i, j) \in I$ satisfying $\alpha_{i + n_{q} j} = c_{k}$. 
Therefore, from $\bm{\alpha} \in (\mathbb{Z}_{\geq 0})^{n}$, the lemma is obtained. 
\end{proof} 

\begin{lem}\label{lem:17}
For any $0 \leq i \leq n_{p q} - 1$ and $0 \leq j \leq p - 2$, 
it holds that $c^{(i)} + d_{i + n_{p q} j} \in \mathbb{Z}_{\geq 0}$. 
\end{lem}
\begin{proof}
Fix any $0 \leq i \leq n_{p q} - 1$ and $0 \leq j \leq p - 2$. 
It follows from $\bm{\alpha} \in (\mathbb{Z}_{\geq 0})^{n}$ that $c_{k} + d_{i + n_{p q} j} = \alpha_{i + n_{p q} j + n_{q} l} \in \mathbb{Z}_{\geq 0}$ for any $0 \leq l \leq q - 1$, where $k$ is the reminder when $i + n_{p q} j + n_{q} l$ is divided by $n_{p}$. 
Thus, it holds that $\tilde{c}^{(i)} + d_{i + n_{p q} j} \in \mathbb{Z}_{\geq 0}$, where 
\begin{align*}
\tilde{c}^{(i)} := \min{\{c_{k} \mid k \equiv i + n_{p q} j + n_{q} l \!\!\! \pmod{n_{p}},\; 0 \leq l \leq q - 1,\; 0 \leq k \leq n_{p} - 1 \} }. 
\end{align*}
We prove $\tilde{c}^{(i)} = c^{(i)}$. 
For the purpose, it suffices to show that $I_{i} = J_{i}$, where 
\begin{align*}
I_{i} &:= \{ k \mid k \equiv i + n_{p q} j + n_{q} l \!\!\! \pmod{n_{p}},\; 0 \leq l \leq q - 1,\; 0 \leq k \leq n_{p} - 1 \}, \\ 
J_{i} &:= \{ k \mid k \equiv i \!\!\! \pmod{n_{p q}},\; 0 \leq k \leq n_{p} - 1 \}. 
\end{align*}
It is readily verified that $I_{i} \subset J_{i}$. 
We demonstrate $I_{i} \supset J_{i}$. 
Let $k = i + n_{p q} j' \in J_{i}$. 
From $\gcd(p, q) = 1$, it follows that there exist integers $a$ and $b$ satisfying $j' - j = p a + q b$, which implies $k = i + n_{p q} j + n_{q} a + n_{p} b$. 
Thus, by choosing the integers $m$ and $r$ with $0 \leq r \leq q - 1$ satisfying $a = q m + r$, we have $k = i + n_{p q} j + n_{q} (q m + r) + n_{p} b = i + n_{p q} j + n_{q} r + n_{p} (p m + b) \in I_{i}$. 
\end{proof}

\subsection{Proof of Theorem~$\mathbf{\ref{thm:2}}$}
We prove Theorem~$\ref{thm:2}$. 
For any $0 \leq i < i' \leq n_{p} - 1$, it follows from Lemma~$\ref{lem:13}$ that $\Theta_{p}(\bm{y}^{(p)}_{i})$ and $\Theta_{p}(\bm{y}^{(p)}_{i'})$ have no common variables $y_{j}$, so $\Theta_{p}(\bm{y}^{(p)}_{0}), \Theta_{p}(\bm{y}^{(p)}_{1}), \ldots, \Theta_{p}(\bm{y}^{(p)}_{n_{p} - 1})$ can be regarded as $n_{p}$ independent variables. 
From this, we have $\Ker(\rho) = \{0 \}$. 

Let $\mathfrak{I} := (t_{0}, t_{1}, \ldots, t_{n_{p q} - 1})$ be the ideal of $\mathbb{C}[z_{i}, w_{j} \,\vert\, 0 \leq i \leq n_{p} - 1,\, 0 \leq j \leq n_{q} - 1]$. 
We prove $\Ker(\rho') = \mathfrak{I}$. 
From Lemmas~$\ref{lem:10}$ and $\ref{lem:13}$, it holds that $t_{0}, t_{1}, \ldots, t_{n_{p q} - 1} \in \Ker(\rho')$. 
Thus, we have $\Ker(\rho') \supset \mathfrak{I}$. 
Below, we prove $\Ker(\rho') \subset \mathfrak{I}$. 
Let $\bm{z} := {}^{t}(z_{0}, z_{1}, \ldots, z_{n_{p} - 1})$ and $\bm{w} := {}^{t}(w_{0}, w_{1}, \ldots, w_{n_{q} - 1})$. 
Suppose now that 
\[
\sum_{(\bm{c}, \bm{d}) \in S} a_{(\bm{c}, \bm{d})} \bm{z}^{\bm{c}} \bm{w}^{\bm{d}} \in \Ker(\rho')
\]
for some $S \subset (\mathbb{Z}_{\geq 0})^{n_{p}} \times (\mathbb{Z}_{\geq 0})^{n_{q}}$ and $a_{(\bm{c}, \bm{d})} \in \mathbb{C}$. 
Then, from Lemma~$\ref{lem:13}$, we have 
\begin{align*}
0 = \rho' \left(\sum_{(\bm{c}, \bm{d}) \in S} a_{(\bm{c}, \bm{d})} \bm{z}^{\bm{c}} \bm{w}^{\bm{d}} \right) 
= \sum_{(\bm{c}, \bm{d}) \in S} a_{(\bm{c}, \bm{d})} \bm{y}^{\sigma (\bm{c}, \bm{d})} 
= \sum_{\bm{\alpha} \in V_{S}} \sum_{(\bm{c}, \bm{d}) \in S_{\bm{\alpha}}} a_{(\bm{c}, \bm{d})} \bm{y}^{\bm{\alpha}}, 
\end{align*}
where $V_{S} := \{\sigma (\bm{c}, \bm{d}) \in (\mathbb{Z}_{\geq 0})^{n} \mid (\bm{c}, \bm{d}) \in S \}$ and $S_{\bm{\alpha}} := \{(\bm{c}, \bm{d}) \in S \mid \sigma (\bm{c}, \bm{d}) = \bm{\alpha} \}$. 
Comparing the coefficient of $\bm{y}^{\bm{\alpha}}$ gives the following lemma. 

\begin{lem}\label{lem:18}
For any $\bm{\alpha} \in V_{S}$, it holds that $\sum_{(\bm{c}, \bm{d}) \in S_{\bm{\alpha}}} a_{(\bm{c}, \bm{d})} = 0$. 
\end{lem}

Also, from Lemma~$\ref{lem:12}$, we obtain the following lemma. 
\begin{lem}\label{lem:19}
For any $(\bm{c}, \bm{d})$, $(\bm{c}', \bm{d}') \in S_{\bm{\alpha}}$, 
it holds that $\bm{z}^{\bm{c}} \bm{w}^{\bm{d}} - \bm{z}^{\bm{c}'} \bm{w}^{\bm{d}'} \in \mathfrak{I}$. 
\end{lem}
\begin{proof}
%The case when $(\bm{c}, \bm{d}) = (\bm{c}', \bm{d}')$ is trivial. 
Let $\bm{c} := (c_{0}, c_{1}, \ldots, c_{n_{p} - 1})$, $\bm{d} := (d_{0}, d_{1}, \ldots, d_{n_{q} - 1})$, $\bm{c}' := (c'_{0}, c'_{1}, \ldots, c'_{n_{p} - 1})$, and $\bm{d}' := (d'_{0}, d'_{1}, \ldots, d'_{n_{q} - 1})$. 
Suppose that $(\bm{c}, \bm{d})$, $(\bm{c}', \bm{d}') \in S_{\bm{\alpha}}$. 
%and $(\bm{c}, \bm{d}) \neq (\bm{c}', \bm{d}')$. 
Then, since $(\bm{c} - \bm{c}', \bm{d} - \bm{d}') \in \Ker (\sigma)$ holds, 
it follows from Lemma~$\ref{lem:12}$ that there exists $\bm{m} := (m_{0}, m_{1}, \ldots, m_{n_{p q} - 1}) \in \mathbb{Z}^{n_{p q}}$ satisfying 
\[
\bm{c} - \bm{c}' = (\underbrace{\bm{m}, \bm{m}, \ldots, \bm{m}}_{q}), \quad 
\bm{d} - \bm{d}' = - (\underbrace{\bm{m}, \bm{m}, \ldots, \bm{m}}_{p}). 
\]
We define $I_{+} := \{i \mid m_{i} \geq 0,\, 0 \leq i \leq n_{p q} - 1 \}$ and $I_{-} := \{i \mid m_{i} < 0,\, 0 \leq i \leq n_{p q} - 1 \}$. 
From the above, we have 
\begin{align*}
\bm{z}^{\bm{c}'} 
&= \prod_{i = 0}^{n_{p q} - 1} \prod_{j = 0}^{q - 1} z_{i + n_{p q} j}^{c'_{i + n_{p q} j}} 
%= \left(\prod_{i \in I_{+}} \prod_{j = 0}^{q - 1} z_{i + n_{p q} j}^{c'_{i + n_{p q} j}} \right) \left(\prod_{i \in I_{-}} \prod_{j = 0}^{q - 1} z_{i + n_{p q} j}^{c_{i + n_{p q} j} - m_{i}} \right) 
= \left(\prod_{i \in I_{+}} \prod_{j = 0}^{q - 1} z_{i + n_{p q} j}^{c'_{i + n_{p q} j}} \right) \left(\prod_{i \in I_{-}} \prod_{j = 0}^{q - 1} z_{i + n_{p q} j}^{c_{i + n_{p q} j}} \right) \left(\prod_{i \in I_{-}} u_{i}^{-m_{i}} \right), \\
\bm{w}^{\bm{d}'} 
&= \prod_{i = 0}^{n_{p q} - 1} \prod_{j = 0}^{p - 1} w_{i + n_{p q} j}^{d'_{i + n_{p q} j}} 
%= \left(\prod_{i \in I_{+}} \prod_{j = 0}^{p - 1} w_{i + n_{p q} j}^{d_{i + n_{p q} j} + m_{i}} \right) \left(\prod_{i \in I_{-}} \prod_{j = 0}^{p - 1} w_{i + n_{p q} j}^{d'_{i + n_{p q} j}} \right) 
= \left(\prod_{i \in I_{-}} \prod_{j = 0}^{p - 1} w_{i + n_{p q} j}^{d'_{i + n_{p q} j}} \right) \left(\prod_{i \in I_{+}} \prod_{j = 0}^{p - 1} w_{i + n_{p q} j}^{d_{i + n_{p q} j}} \right) \left(\prod_{i \in I_{+}} v_{i}^{m_{i}} \right), 
\end{align*}
where $u_{i} := \prod_{j = 0}^{q - 1} z_{i + n_{p q} j}$ and $v_{i} := \prod_{j = 0}^{p - 1} w_{i + n_{p q} j}$. 
Since $t_{i} = u_{i} - v_{i}$ holds for $0 \leq i \leq n_{p q} - 1$, there exist $t, t' \in \mathfrak{I}$ satisfying 
\begin{align*}
\prod_{i \in I_{-}} u_{i}^{- m_{i}} 
&= \prod_{i \in I_{-}} (v_{i} + t_{i})^{- m_{i}} 
= \prod_{i \in I_{-}} v_{i}^{- m_{i}} + t, \\
\prod_{i \in I_{+}} v_{i}^{m_{i}}
&= \prod_{i \in I_{+}} (u_{i} - t_{i})^{m_{i}} 
= \prod_{i \in I_{+}} u_{i}^{m_{i}} + t'. 
\end{align*}
Therefore, there exists $t'' \in \mathfrak{I}$ satisfying 
\begin{align*}
\bm{z}^{\bm{c}'} \bm{w}^{\bm{d}'} 
%&= \left(\prod_{i \in I_{+}} \prod_{j = 0}^{q - 1} z_{i + n_{p q} j}^{c'_{i + n_{p q} j}} \right) \left(\prod_{i \in I_{-}} \prod_{j = 0}^{q - 1} z_{i + n_{p q} j}^{c_{i + n_{p q} j}} \right) \left(\prod_{i \in I_{-}} u_{i}^{-m_{i}} \right) \\
%&\quad \times \left(\prod_{i \in I_{-}} \prod_{j = 0}^{p - 1} w_{i + n_{p q} j}^{d'_{i + n_{p q} j}} \right) \left(\prod_{i \in I_{+}} \prod_{j = 0}^{p - 1} w_{i + n_{p q} j}^{d_{i + n_{p q} j}} \right) \left(\prod_{i \in I_{+}} v_{i}^{m_{i}} \right) \\
&= \left(\prod_{i \in I_{+}} \prod_{j = 0}^{q - 1} z_{i + n_{p q} j}^{c'_{i + n_{p q} j}} \right) \left(\prod_{i \in I_{-}} \prod_{j = 0}^{q - 1} z_{i + n_{p q} j}^{c_{i + n_{p q} j}} \right) \left(\prod_{i \in I_{-}} v_{i}^{-m_{i}} \right) \\
&\quad \times 
\left(\prod_{i \in I_{-}} \prod_{j = 0}^{p - 1} w_{i + n_{p q} j}^{d'_{i + n_{p q} j}} \right) \left(\prod_{i \in I_{+}} \prod_{j = 0}^{p - 1} w_{i + n_{p q} j}^{d_{i + n_{p q} j}} \right) \left(\prod_{i \in I_{+}} u_{i}^{m_{i}} \right) + t'' \\
&= \left(\prod_{i \in I_{+}} \prod_{j = 0}^{q - 1} z_{i + n_{p q} j}^{c'_{i + n_{p q} j}} \right) \left(\prod_{i \in I_{+}} u_{i}^{m_{i}} \right) 
\left(\prod_{i \in I_{-}} \prod_{j = 0}^{q - 1} z_{i + n_{p q} j}^{c_{i + n_{p q} j}} \right) \\
&\quad \times 
\left(\prod_{i \in I_{-}} \prod_{j = 0}^{p - 1} w_{i + n_{p q} j}^{d'_{i + n_{p q} j}} \right) \left(\prod_{i \in I_{-}} v_{i}^{-m_{i}} \right) 
\left(\prod_{i \in I_{+}} \prod_{j = 0}^{p - 1} w_{i + n_{p q} j}^{d_{i + n_{p q} j}} \right) + t'' \\
&= \bm{z}^{\bm{c}} \bm{w}^{\bm{d}} + t''. 
\end{align*}
Thus the lemma is proved. 
\end{proof}

The above two lemmas give the following lemma. 

\begin{lem}\label{lem:20}
For any $\bm{\alpha} \in V_{S}$, it holds that $\sum_{(\bm{c}, \bm{d}) \in S_{\bm{\alpha}}} a_{(\bm{c}, \bm{d})} \bm{z}^{\bm{c}} \bm{w}^{\bm{d}} \in \mathfrak{I}$. 
\end{lem}
\begin{proof}
Let $\bm{\alpha} \in V_{S}$ and fix any $(\bm{c}', \bm{d}') \in S_{\bm{\alpha}}$. 
Then, from Lemma~$\ref{lem:19}$, there exists $t_{\bm{\alpha}} \in \mathfrak{I}$ satisfying 
\[
\sum_{(\bm{c}, \bm{d}) \in S_{\bm{\alpha}}} a_{(\bm{c}, \bm{d})} \bm{z}^{\bm{c}} \bm{w}^{\bm{d}} 
= \left(\sum_{(\bm{c}, \bm{d}) \in S_{\bm{\alpha}}} a_{(\bm{c}, \bm{d})} \right) \bm{z}^{\bm{c}'} \bm{w}^{\bm{d}'} 
+ t_{\bm{\alpha}}.
\]
Thus, from Lemma~$\ref{lem:18}$, we have $\sum_{(\bm{c}, \bm{d}) \in S_{\bm{\alpha}}} a_{(\bm{c}, \bm{d})} \bm{z}^{\bm{c}} \bm{w}^{\bm{d}} 
= t_{\bm{\alpha}} \in \mathfrak{I}$. 
\end{proof}

It follows from Lemma~$\ref{lem:20}$ that 
\[
\sum_{(\bm{c}, \bm{d}) \in S} a_{(\bm{c}, \bm{d})} \bm{z}^{\bm{c}} \bm{w}^{\bm{d}} 
= \sum_{\bm{\alpha} \in V_{S}} \sum_{(\bm{c}, \bm{d}) \in S_{\bm{\alpha}}} a_{(\bm{c}, \bm{d})} \bm{z}^{\bm{c}} \bm{w}^{\bm{d}} 
\in \mathfrak{I}. 
\]
This completes the proof of Theorem~$\ref{thm:2}$.

\subsection{Proof of Theorem~$\mathbf{\ref{thm:3}}$}
We prove Theorem~$\ref{thm:3}$. 
Since $\mathbb{C}[\bm{x}]^{D} = \mathbb{C}[\bm{x}]^{\Delta} \supset R_{n}$ holds from Lemma~$\ref{lem:14}$, it only remains to prove $\mathbb{C}[\bm{x}]^{D} \neq R_{n}$. 
From Lemma~$\ref{lem:9}$, we can take $\bm{\alpha}' \in (V_{n} \cap (\mathbb{Z}_{\geq 0})^{n}) \setminus \mathbb{Z}_{\geq 0}\mathchar`-\Span T_{n}$. 
Then, as mentioned in Section~$2.1$, it holds that $\bm{y}^{\bm{\alpha}'} \in \mathbb{C}[\bm{x}]^{D}$. 
On the other hand, it holds that $\bm{y}^{\bm{\alpha}'} \notin R_{n}$. 
To verify this, let 
\[
\psi : \mathbb{C}[z^{(p)}_{i} \mid \text{$p$ is a prime factor of $n$},\; 0 \leq i \leq n_{p} - 1] \to \mathbb{C}[\bm{y}]
\]
be the ring homomorphism defined by $\psi (z^{(p)}_{i}) := \bm{y}^{{\bm v}^{(p)}_{i}}$. 
Note that from Lemma~$\ref{lem:13}$, it follows that 
\[
R_{n} = \mathbb{C}[\bm{y}^{{\bm v}^{(p)}_{i}} \mid \text{$p$ is a prime factor of $n$},\; 0 \leq i \leq n_{p} - 1]. 
\]
We prove $\bm{y}^{\bm{\alpha}'} \notin \Ima(\psi) = R_{n}$ by contradiction. 
Let $n = p_{1}^{l_1} p_{2}^{l_2} \cdots p_{k}^{l_{k}}$ be the prime factorization of $n$, where $k \geq 3$ and $l_j\geq 1$, and 
let $\bm{z} := {}^{t}(\bm{z}_{p_1}, \bm{z}_{p_2}, \ldots, \bm{z}_{p_k})$ with $\bm{z}_{p_j} := {}^{t}(z^{(p_j)}_{0}, z^{(p_j)}_{1}, \ldots, z^{(p_j)}_{n_{p_j} - 1})$. 
Assume that 
\[
\psi \left(\sum_{\bm{c} \in S} a_{\bm{c}} \bm{z}^{\bm{c}} \right) = \bm{y}^{\bm{\alpha}'}
\]
for some $S \subset (\mathbb{Z}_{\geq 0})^{n_{p_1} + n_{p_2} + \cdots + n_{p_k}}$ and $a_{\bm{c}} \in \mathbb{C}$. 
Let $\tau : \mathbb{Z}^{n_{p_1} + n_{p_2} + \cdots + n_{p_k}} \to \mathbb{Z}^{n}$ be the homomorphism of the $\mathbb{Z}$-modules defined by 
\[
\tau (\bm{c}) := \sum_{j = 1}^{k} \sum_{i = 0}^{n_{p_j} - 1} c^{(p_j)}_{i} {\bm v}^{(p_j)}_{i}, 
\]
where $\bm{c} := (\bm{c}_{p_1}, \bm{c}_{p_2}, \ldots, \bm{c}_{p_k})$ with $\bm{c}_{p_j} := (c^{(p_j)}_{0}, c^{(p_j)}_{1}, \ldots, c^{(p_j)}_{n_{p_j} - 1})$. 
Then, from the assumption, we have 
\begin{align*}
\bm{y}^{\bm{\alpha}'} 
= \psi \left(\sum_{\bm{c} \in S} a_{\bm{c}} \bm{z}^{\bm{c}} \right) 
= \sum_{\bm{c} \in S} a_{\bm{c}} \bm{y}^{\tau (\bm{c})} 
= \sum_{\bm{\alpha} \in V_{S}} \sum_{\substack{\bm{c} \in S \\ \tau (\bm{c}) = \bm{\alpha}}} a_{\bm{c}} \bm{y}^{\bm{\alpha}}, 
\end{align*}
where $V_{S} := \{\tau (\bm{c}) \in (\mathbb{Z}_{\geq 0})^{n} \mid \bm{c} \in S \}$. 
From this, it follows that $\bm{\alpha}' \in V_{S} \subset \mathbb{Z}_{\geq 0}\mathchar`-\Span T_{n}$. 
This is a contradiction. 
Thus, we conclude that $\bm{y}^{\bm{\alpha}'} \notin \Ima(\psi) = R_{n}$.

\subsection{Proof of Theorem~$\mathbf{\ref{thm:4}}$}
Theorem~$\ref{thm:4}$ is proved from  
\[
\Theta_{n}(\bm{x}) = \bm{y}^{(1, 1, \ldots, 1)} = \bm{y}^{\sum_{i = 0}^{n_{p} - 1} {\bm v}^{(p)}_{i}} = \prod_{i = 0}^{n_{p} - 1} \bm{y}^{{\bm v}^{(p)}_{i}}
\]
and Lemma~$\ref{lem:13}$.

\subsection{Proof of Theorem~$\mathbf{\ref{thm:5}}$}
We prove Theorem~$\ref{thm:5}$. 
Let $Q$ be the unitary matrix defined as $Q := \frac{1}{\sqrt{n}} (\zeta_{n}^{ij})_{0 \leq i, j \leq n - 1}$ and let $P := \frac{1}{\sqrt{n}} {}^{t} \overline{Q} = \frac{1}{n} (\zeta_{n}^{- i j})_{0 \leq i, j \leq n - 1}$. 
Then, from $\bm{y} = \sqrt{n} Q \bm{x}$, we have $\bm{x} = P \bm{y}$. 
Note that from the equality 
\begin{align*}
\left\{ P^{-1} A P \mid A \in \SL(\mathbb{C}G) \right\} = \left\{ \diag(\bm{\lambda}) \mid \bm{\lambda} = (\lambda_{0}, \lambda_{1}, \ldots, \lambda_{n - 1}) \in \mathbb{C}^{n}, \;\; \lambda_{0} \lambda_{1} \cdots \lambda_{n - 1} = 1 \right\}, 
\end{align*}
where $\diag (\bm{\lambda})$ denotes the diagonal matrix constructed from a vector $\bm{\lambda}$, 
the following holds: Let $\bm{\alpha} = (\alpha_{0}, \alpha_{1}, \ldots, \alpha_{n-1}) \in (\mathbb{Z}_{\geq 0})^{n}$, then   
\begin{align*}
&(P^{-1} A P \bm{y})^{\bm{\alpha}} = \bm{y}^{\bm{\alpha}} \;\; 
\text{for any}\;\; A \in \SL(\mathbb{C}G) \\
&\quad \iff \lambda_{0}^{\alpha_{0}} \lambda_{1}^{\alpha_{1}} \cdots \lambda_{n-1}^{\alpha_{n-1}} = 1 \;\; 
\text{for any} \;\; (\lambda_{0}, \lambda_{1}, \ldots, \lambda_{n-1}) \in \mathbb{C}^{n} \;\; 
\text{with} \;\; \lambda_{0} \lambda_{1} \cdots \lambda_{n-1} = 1 \\
&\quad \iff \alpha_{0} = \alpha_{1} = \cdots = \alpha_{n - 1}. 
\end{align*}
For any $f(\bm{x}) \in \mathbb{C}[\bm{x}]$, 
we define $g(\bm{y}) \in \mathbb{C}[\bm{y}]$ by $g(\bm{y}) := f(P \bm{y}) = f(\bm{x})$. 
Then, we have 
\begin{align*}
f(\bm{x}) \in \mathbb{C}[\bm{x}]^{\SL(\mathbb{C}G)} 
&\iff f(A \bm{x}) = f(\bm{x}) \;\; 
\text{for any} \;\; A \in \SL(\mathbb{C}G) \\ 
&\iff g(P^{-1} A P \bm{y}) = g(\bm{y}) \;\; 
\text{for any} \;\; A \in \SL(\mathbb{C}G) \\ 
&\iff g(\bm{y}) = \sum_{i = 0}^{m} c_{i} \bm{y}^{(i, i, \ldots, i)} \;\; 
\text{for some} \;\; c_{0}, c_{1}, \ldots, c_{m} \in \mathbb{C} \\ 
&\iff g(\bm{y}) \in \mathbb{C}[\bm{y}^{(1, 1, \ldots, 1)}] = \mathbb{C}[\Theta_{n}(\bm{x})] \\ 
&\iff f(\bm{x}) \in \mathbb{C}[\Theta_{n}(\bm{x})]. 
\end{align*}
Thus, it holds that $\mathbb{C}[\bm{x}]^{\SL(\mathbb{C}G)} = \mathbb{C}[\Theta_{n}(\bm{x})]$. 
Also, from Theorem~$\ref{thm:4}$ and Lemma~$\ref{lem:14}$, we obtain $\mathbb{C}[\Theta_{n}(\bm{x})] \subset R_{n} \subset \mathbb{C}[\bm{x}]^{D}$.

\clearpage

\bibliography{reference}

\begin{thebibliography}{10}

\bibitem{MR894702}
Yoshio Agaoka.
\newblock On {C}ayley-{H}amilton's theorem and {A}mitsur-{L}evitzki's identity.
\newblock {\em Proc. Japan Acad. Ser. A Math. Sci.}, 63(3):82--85, 1987.

\bibitem{MR1137542}
Dragomir~Z. Djokovic.
\newblock On the product of two alternating matrices.
\newblock {\em Amer. Math. Monthly}, 98(10):935--936, 1991.

\bibitem{Frobenius1968gruppen}
Ferdinand~Georg Frobenius.
\newblock \"{U}ber die {P}rimfactoren der {G}ruppendeterminante.
\newblock {\em Sitzungsberichte der K\"{o}niglich Preu{\ss}ischen Akademie der
  Wissenschaften zu Berlin}, pages 1343--1382, 1896.
\newblock Reprinted in {\it Gesammelte Abhandlungen, Band III}. Springer-Verlag
  Berlin Heidelberg, New York, 1968, pages 38--77.

\bibitem{MR0446837}
Thomas Hawkins.
\newblock New light on {F}robenius' creation of the theory of group characters.
\newblock {\em Arch. History Exact Sci.}, 12:217--243, 1974.

\bibitem{MR1848969}
Minoru Itoh.
\newblock A {C}ayley-{H}amilton theorem for the skew {C}apelli elements.
\newblock {\em J. Algebra}, 242(2):740--761, 2001.

\bibitem{MR3436396}
Minoru Itoh.
\newblock Invariant theory in exterior algebras and {A}mitsur-{L}evitzki type
  theorems.
\newblock {\em Adv. Math.}, 288:679--701, 2016.

\bibitem{MR803326}
B.~L. van~der Waerden.
\newblock {\em A history of algebra}.
\newblock Springer-Verlag, Berlin, 1985.
\newblock From al-Khw\={a}rizm\={\i} to Emmy Noether.

\bibitem{MR1488158}
Hermann Weyl.
\newblock {\em The classical groups}.
\newblock Princeton Landmarks in Mathematics. Princeton University Press,
  Princeton, NJ, 1997.
\newblock Their invariants and representations, Fifteenth printing, Princeton
  Paperbacks.

\bibitem{MR4526227}
Naoya Yamaguchi and Yuka Yamaguchi.
\newblock Remark on {L}aquer's theorem for circulant determinants.
\newblock {\em Int. J. Group Theory}, 12(4):265--269, 2023.

\bibitem{MR4814714}
Naoya Yamaguchi and Yuka Yamaguchi.
\newblock Generalized {D}edekind's theorem and its application to integer group
  determinants.
\newblock {\em J. Math. Soc. Japan}, 76(4):1123--1138, 2024.

\bibitem{MR4104497}
Yang Zhang.
\newblock The first and second fundamental theorems of invariant theory for the
  quantum general linear supergroup.
\newblock {\em J. Pure Appl. Algebra}, 224(11):106411, 50, 2020.

\end{thebibliography}
\bibliographystyle{plain}

\medskip
\begin{flushleft}
Faculty of Education, 
University of Miyazaki, 
1-1 Gakuen Kibanadai-nishi, 
Miyazaki 889-2192, 
Japan \\ 
{\it Email address}, Naoya Yamaguchi: n-yamaguchi@cc.miyazaki-u.ac.jp \\ 
{\it Email address}, Yuka Yamaguchi: y-yamaguchi@cc.miyazaki-u.ac.jp 
\end{flushleft}

\begin{flushleft}
Institute of Mathematics for Industry, 
Kyushu University, 
744 Motooka, Nishi-ku, 
Fukuoka 819-0395, 
Japan \\ 
{\it Email address}, Hiroyuki Ochiai: ochiai@imi.kyushu-u.ac.jp 
\end{flushleft}

\end{document}